\documentclass[11pt]{amsart}
\usepackage{mathrsfs}
\usepackage{amsfonts}
\usepackage{latexsym,amsmath,amssymb}

 \renewcommand{\epsilon}{\varepsilon}

 \newtheorem{theorem}{Theorem}[section]
 
 \newtheorem{lemma}[theorem]{Lemma}
 
 \newtheorem{Corollary}[theorem]{Corollary}
 
 \newtheorem{proposition}[theorem]{Proposition}

 \newtheorem{deff}[theorem]{Definition}
 \newtheorem{rem}[theorem]{Remark}
 \newcommand{\bth}{\begin{theorem}}
 \newcommand{\ble}{\begin{lemma}}
 \newcommand{\bcor}{\begin{corr}}
 \newcommand{\bdeff}{\begin{deff}}
 \newcommand{\bprop}{\begin{proposition}}
 \newcommand{\ele}{\end{lemma}}
 \newcommand{\ecor}{\end{corr}}
 \newcommand{\edeff}{\end{deff}}
 
 \newcommand{\eprop}{\end{proposition}}

 \renewcommand{\Pi}{\varPi}

 \renewcommand{\epsilon}{\varepsilon}

\numberwithin{equation}{section}

\title
[On the entire self-shrinking solutions]{On the entire
self-shrinking solutions
 to Lagrangian Mean Curvature Flow}
\author{Rongli Huang}
\author{Zhizhang Wang}
\address{School of Mathematical Science, Fudan University, Shanghai 200433, People's Republic of China, E-mail: huangronglijane@yahoo.cn}
\address{School of Mathematical Science, Fudan University, Shanghai 200433, People's Republic of China, E-mail: youxiang163wang@163.com }
\date{}

\begin{document}
\maketitle

\begin{abstract}
The authors prove that the logarithmic Monge-Amp\`{e}re  flow with uniformly
bound and convex initial data satisfies uniform decay estimates away
from time $t=0$. Then applying the decay estimates, we conclude  that
every entire classical strictly convex solution of
 the equation
 \begin{equation*}
\det D^{2}u=\exp\{n(-u+\frac{1}{2}\sum_{i=1}^{n}x_{i}\frac{\partial u}{\partial x_{i}})\},
\end{equation*}
 should be a quadratic polynomial if the inferior limit of the smallest eigenvalue
 of the function $|x|^{2}D^{2}u$ at infinity has an uniform positive lower bound  larger than $2(1-1/n)$.
 Using a similar method, we can prove that every classical convex or concave solution
 of the equation
 \begin{equation*}
\sum_{i=1}^{n}\arctan\lambda_{i}=-u+\frac{1}{2}\sum_{i=1}^{n}x_{i}\frac{\partial u}{\partial x_{i}}.
\end{equation*}
must be a quadratic polynomial, where $\lambda_{i}$ are the eigenvalues
of  the Hessian $D^{2}u$.

 \end{abstract}



\section{Introduction}
In 1915, S. Bernstein  \cite{Nit} proved his celebrated theorem that
the only entire minimal graphs in 3 dimensional Euclidean space are
planes. In 1954, K. J\"{o}rgens  \cite{J} proved that every
classical strictly convex solution of the equation
\begin{equation}
\det D^{2}u=1,\,\,\, x\in \mathbb{R}^{2}
\end{equation}
must be a quadratic polynomial, and  Bernstein's theorem can be
proved  using this result. Meanwhile, E. Calabi ($n\leq5$) \cite{C}
and A.V. Pogorelov ($n\geq2$) \cite{P} extended K. J\"{o}rgens'
theorem to $\mathbb{R}^n$. Later J. Jost and Y.L. Xin had an
alternative proof for this result \cite{JX}. In 2003, L. Cafarelli
and Y.Y. Li \cite{LL} gave an extension  to the theorem of K.
J\"{o}rgens, E. Calabi and A.V. Pogorelov. In that paper, they presented
another proof of the theorem of K. J\"{o}rgens, E. Calabi and A.V.
Pogorelov and did research on the asymptotic behavior of  convex
solutions. They also used their results to reprove the Bernstein
theorems. Recently, A.M. Li and R.W. Xu \cite{LX} showed that every
smooth strictly convex solution on $\mathbb{R}^{n}$ of the
Monge-Amp\`{e}re  type equation
\begin{equation}\label{e1.2}
\det D^{2}u=\exp\{-\sum_{i=1}^{n}d_{i}\frac{\partial u}{\partial x_{i}}-d_{0}\},\,\,\, x\in \mathbb{R}^{n}
\end{equation}
must be a quadratic polynomial where $d_{0}, d_{1},\cdots,d_{n}$ are
constants.

From \cite{HL}, we know that the gradient graph $(x, \nabla u)$
determines a volume minimizing surface in $\mathbb{C}^{n}$ if and
only if $u$ satisfies the special Lagrangian equation
\begin{equation*}
 \sum_{i=1}^{n}\arctan\lambda_{i}=\Theta.
\end{equation*}
Here, $\lambda_{i}$ are the eigenvalues of  the Hessian $D^{2}u$ and
$\Theta$ is a constant. It belongs to an important class of fully
nonlinear elliptic equations which has been studied by various
authors (cf. J.G. Bao, J.Y. Chen, B. Guan, M. Jin \cite{BC}, Y.
Yuan \cite{Y}). A Bernstein type theorem has been proved in \cite{Y}.
It tells us, that if $u$ is a smooth convex function and satisfies the special Lagrangian equation
in $\mathbb{R}^{n}$  then $u$ must be a quadratic polynomial.

Lagrangian self-similar solution being part of a minimal cone was investigated in \cite{KK} with additional conditions
on Maslov class and the Lagrangian angle.
In this paper we mainly do research  on a Bernstein type problem  for
self-shrinking equations of the  Lagrangian mean curvature flow in
Euclidean and pseudo-Euclidean space.

Consider the logarithmic Monge-Amp\`{e}re  flow, (cf.\cite{KT})
\begin{equation}\label{e1.6}
\left\{ \begin{aligned}\frac{\partial u}{\partial t}-\frac{1}{n}\ln
\det D^{2}u&=0,
& t>0,\quad x\in \mathbb{R}^{n}, \\
 u&=u_{0}(x), & t=0,\quad x\in \mathbb{R}^{n}.
\end{aligned} \right.
\end{equation}
By Proposition 2.1 in \cite{HB}, there exists a family of diffeomorphisms
$$ r_{t}: \mathbb{R}^{n}\rightarrow \mathbb{R}^{n},$$
such that the maps
\begin{equation*}\begin{aligned}
F(x,t)&=(r_{t}(x),  Du(r_{t}(x),t)) \subset
\mathbb{R}^{2n}_{n},\\
F_{0}(x)&=(x, Du_{0}(x)).
\end{aligned}
\end{equation*}
satisfy the mean curvature flow in pseudo-Euclidean
space:
\begin{equation}\label{e1.7}
\left\{ \begin{aligned}\frac{dF}{dt}&=\overrightarrow{H}, \\
F(x,0)&=F_{0}(x),
\end{aligned} \right. \end{equation}
where $\overrightarrow{H}$ is the mean curvature vector of the
sub-manifold defined by $F$.

\begin{deff}
Assume  that function $u_{0}(x)\in C^{2}(\mathbb{R}^{n})$. We call
$u_{0}(x)$ satisfying Condition A, if
\begin{equation*}
\Lambda I\geq D^{2}u_{0}(x)\geq \lambda I,\qquad x\in
\mathbb{R}^{n}.
\end{equation*}
Here  $ \Lambda, \lambda$ are two positive constants
 and $I$ is the identity matrix.
\end{deff}

For the logarithmic Monge-Amp\`{e}re  flow, the first author  has obtained
the long time existence and the global estimates of
derivatives of solutions (cf. Theorem 1.2 in \cite{HB}).

\begin{proposition}\label{p3.1}
Let $u_{0}:\mathbb{R}^{n}\rightarrow \mathbb{R}$ be a $C^{2}$
function which satisfies  Condition A.
Then  there exists a unique strictly
convex solution  of (\ref{e1.6}) such that
\begin{equation}\label{e3.1}
u(x,t)\in C^{\infty}(\mathbb{R}^{n}\times (0,+\infty))\cap
C(\mathbb{R}^{n}\times [0,+\infty))
\end{equation}
where $u(\cdot,t)$ satisfies Condition A.  More generally, for  $l\in\{3,4,5\cdots\}$ and $\epsilon_{0}>0$,  there holds
\begin{equation}\label{e3.3a}
 \sup_{x\in\mathbb{R}^{n}}|D^{l}u(x,t)|^{2}\leq
 C, \qquad \forall t\in (\epsilon_{0},+\infty),
\end{equation}
where $C$  depends only on $n, \lambda, \Lambda, \dfrac{1}{\epsilon_{0}}.$
\end{proposition}

In fact, in this paper we will prove the following  stronger result:
\begin{theorem}\label{t1.2a}
Assume that   $u(x,t)$ is a strictly convex  solution of
(\ref{e1.6}), and $u(\cdot,t)$ satisfies  Condition A.
Then there exists a positive constant $C$  depending only on $n,
\lambda, \Lambda, \dfrac{1}{\epsilon_{0}}$, such that
\begin{equation}\label{e1.6a}
\sup_{x\in\mathbb{R}^{n}}|D^{3}u(x,t)|^{2}\leq
\frac{C}{t}, \qquad \forall t\geq \epsilon_{0}.
\end{equation}
More generally, for all $l\in\{3,4,5\cdots\}$ there holds
\begin{equation}\label{e1.7a}
\sup_{x\in\mathbb{R}^{n}}|D^{l}u(x,t)|^{2}\leq
\frac{C}{t^{l-2}}, \qquad \forall t\geq\epsilon_{0}.
\end{equation}
\end{theorem}
\begin{rem}
For the special Lagrangian evolution equation (\ref{e1.6b}),
there are similar results in the paper \cite{ACH1}.

\end{rem}

Next we consider the following  Monge-Amp\`{e}re type equation
\begin{equation}\label{e1.3} \det
D^{2}u=\exp\{n(-u+\frac{1}{2}\sum_{i=1}^{n}x_{i}\frac{\partial
u}{\partial x_{i}})\}.
\end{equation}
According to definitions in \cite{CM}, we can show that an entire
solution to (\ref{e1.3})  is  a self-shrinking solution to
Lagrangian mean curvature flow in Pseudo-Euclidean space.  As an
application of Proposition \ref{p3.1} and Theorem \ref{t1.2a}, we
can prove that
\begin{theorem}\label{t1.1}
Assume that $u :\mathbb{R}^{n}\rightarrow \mathbb{R}$ is a $C^{2}$ strictly convex
solution of (\ref{e1.3}) which satisfies
\begin{equation}\label{e1.511}
\liminf_{x\rightarrow\infty}|x|^{2}\mu(x)>\frac{2(n-1)}{n},
\end{equation}
where $\mu(x)$ is the smallest eigenvalue
 of $D^{2}u$.  Then $u$ must be a quadratic polynomial. Furthermore, there exists a
symmetric real matrix $A$ such that
\begin{eqnarray}\label{a1.6}
u=u(0)+\frac{1}{2}<x,Ax>,
\end{eqnarray}
where $\det A=e^{-nu(0)}$.
\end{theorem}
\begin{rem}
In dimension $1$, assume that $u$ is a smooth solution of (\ref{e1.3})
with $u'(0)=0$, then $$u=u(0)+\frac{1}{2}e^{-nu(0)}x^{2}.$$ This follows from
 the existence and the uniqueness results of ordinary
differential equations.
\end{rem}
By Theorem \ref{t1.1} and by  (\ref{a1.6}) we know that condition (\ref{e1.511}) implies
\begin{eqnarray}\label{e1.12a}
\nabla u(0)=0.
\end{eqnarray}
Then, a natural question is presented. If we weaken condition
(\ref{e1.511}) to (\ref{e1.12a}), does  the same result in Theorem \ref{t1.1} still hold?

In \cite{KM}, the special Lagrangian evolution equation can be
written as \begin{equation}\label{e1.6b} \left\{
\begin{aligned}\frac{\partial u}{\partial t}-
\frac{1}{\sqrt{-1}}\ln\frac{\det(I+\sqrt{-1}D^{2}u)}{\sqrt{\det(I+(D^{2}u)^{2})}}&=0,
& t>0,\quad x\in \mathbb{R}^{n}, \\
 u&=u_{0}(x), & t=0,\quad x\in \mathbb{R}^{n}.
\end{aligned} \right.
\end{equation}
It is well-known that there exists a family of diffeomorphisms
$$ r_{t}: \mathbb{R}^{n}\rightarrow \mathbb{R}^{n},$$
such that
\begin{equation*}\begin{aligned}
F(x,t)&=(r_{t}(x),  Du(r_{t}(x),t)) \subset
\mathbb{R}^{2n},\\
F_{0}(x)&=(x, Du_{0}(x))
\end{aligned}
\end{equation*}
 satisfies the mean curvature flow in
Euclidean space:
\begin{equation}\label{e1.7b}
\left\{ \begin{aligned}\frac{dF}{dt}&=\overrightarrow{H}, \\
F(x,0)&=F_{0}(x), \end{aligned} \right.
\end{equation}
where $\overrightarrow{H}$ is the mean curvature vector of the
sub-manifold defined by $F$.

Consider the entire self-shrinking solutions to Lagrangian mean
curvature flow in  Euclidean space. When the Hessian of the
potential function  $u$ has eigenvalues strictly uniformly between
-1 and 1, A. Chau, J.Y. Chen and W.Y. He   showed that  all
self-shrinking solutions must be quadratic polynomials. The next two
theorems generalize their results \cite{ACH2}.
\begin{theorem}\label{t1.6}
Let $u$ be a $C^{2}$  self-shrinking solution to Lagrangian mean
curvature flow in  Euclidean space:
\begin{equation}\label{e1.9}
\sum_{i=1}^{n}\arctan\lambda_{i}=-u+\frac{1}{2}\sum_{i=1}^{n}x_{i}\frac{\partial
u}{\partial x_{i}},
\end{equation}
where $\lambda_{i}$ $(i=1,2,\cdots,n)$ are the eigenvalues of  Hessian $D^{2}u$.
Suppose that
\begin{equation}\label{e1.10}
-1\leq\lambda_{i}\leq 1,
\end{equation}
then $u$ must be a quadratic polynomial.
\end{theorem}
\begin{theorem}\label{c1.7}
Let $u$ be a $C^{2}$ convex  or concave solution to (\ref{e1.9}).
Then $u$ must be a quadratic polynomial.
\end{theorem}

Here we use some techniques in \cite{Y} and some ideas developed in
the proof of  Lemma \ref{l3.2}. We only use the elliptic equation
(\ref{e1.9}), but  don't need the parabolic equation
(\ref{e1.6b}).

This paper is organized as follows. In section 2, we obtain the
differential inequality (\ref{no1}), which plays an important role
in the third order decay estimates (see Lemma \ref{l3.2a}). Then we
complete the proof of Theorem \ref{t1.2a} by the blow-up argument.
In section 3, we give the proof of Theorem \ref{t1.1} by the second
order derivative estimates for  the equations of Monge-Amp\`{e}re
type (\ref{e1.3}). In section 4,  we  prove Theorem \ref{t1.6}
and \ref{c1.7}.

\section{ the decay estimates of the
 logarithmic Monge-Amp\`{e}re  flow }

 Throughout the following Einstein's convention of
 summation over repeated indices will be adopted.
Denote $$u_{i}=\dfrac{\partial u}{\partial x_{i}},
u_{ij}=\dfrac{\partial^{2}u}{\partial x_{i}\partial x_{j}},
u_{ijk}=\dfrac{\partial^{3}u}{\partial x_{i}\partial x_{j}\partial
x_{k}}, \cdots ,\text{  and  } [u^{ij}]=[ u_{ij}]^{-1}. $$ We
introduce the comparison principle for solutions of Cauchy problems
which belongs to Y. Giga, S. Goto, H. Ishii, M-H. Sato (cf. a
special version of Theorem 4.1 in \cite{Y-M}).
\begin{lemma}\label{l3.1a}
Suppose that the functions $\sigma_{*}, \sigma^{*}\in
C^{2,1}(\mathbb{R}^{n}\times (0,+\infty))\cap C(\mathbb{R}^{n}\times
[0,+\infty))$ and $u$ satisfies Condition A. If there exists a
positive constant $C$, such that
\begin{equation*}
\sigma_{*}\leq C,\quad\sigma^{*}\leq C,
\end{equation*}
and $\sigma_{*}$ , $ \sigma^{*}$ satisfy
\begin{equation*}
\partial_{t}\sigma_{*}-\frac{1}{n}u^{ij}\sigma_{*ij}+\frac{1}{2n^{2}}\sigma_{*}^{2}\leq 0,\quad \forall t>0,\quad x\in \mathbb{R}^{n};
\end{equation*}
\begin{equation*}
\partial_{t}\sigma^{*}-\frac{1}{n}u^{ij}\sigma^{*}_{ij}+\frac{1}{2n^{2}}\sigma^{*2}\geq 0,\quad \forall t>0,\quad x\in \mathbb{R}^{n}; \end{equation*}
\begin{equation*}
\sigma_{*}\leq\sigma^{*}, \quad  t=0,\quad\forall x\in \mathbb{R}^{n}.
\end{equation*}
Then there holds
\begin{equation*}
\sigma_{*}\leq\sigma^{*}, \quad \forall t>0,\quad x\in \mathbb{R}^{n}.
\end{equation*}
\end{lemma}
We are now in a position to describe  Calabi's computation. It is
used by A.V. Pogorelov and L. Caffarelli, L. Nirenbgerg, J. Spruck,
to estimate the third derivatives of Monge-Amp\`{e}re Equation (cf.
\cite{P}, \cite{LLJ1}). Here we use his methods to carry out the
third derivatives of  Monge-Amp\`{e}re Equation of parabolic type.

Let
\begin{equation*}
\sigma=u^{kl}u^{pq}u^{rs}u_{kpr}u_{lqs}.
\end{equation*}
Then the expression measures the square of the third derivatives in terms of
the Riemannian metric $ds^{2}=u_{ij}dx^{i}dx^{i}$. We establish the following lemma which
is a parabolic version of Lemma 3.1  in \cite{LLJ1}.
\begin{lemma}\label{l3.2a}
Let $u$ be a solution of (\ref{e1.6}). If $u(\cdot,t)$ satisfies
(\ref{e3.1}) and Condition A. Then $\sigma$ satisfies a parabolic
inequality:
\begin{equation}\label{no1}
\partial_{t}\sigma-\frac{1}{n}u^{ij}\sigma_{ij}+\frac{1}{2n^{2}}\sigma^{2}\leq 0,\quad \forall t>0,\quad x\in \mathbb{R}^{n}.
\end{equation}
\end{lemma}

\begin{proof}

Note that
\begin{equation*}
 \partial
u^{ab}=-u^{ac}\partial u_{cd}u^{db},
\end{equation*}
\begin{equation*}
\partial_{t}\sigma=2u^{kl}u^{pq}u^{rs}\partial_{t}u_{kpr}u_{lqs}-3u^{ka}\partial_{t}
u_{ab}u^{bl}u^{pq}u^{rs}u_{kpr}u_{lqs}.
\end{equation*}
By the equation (\ref{e1.6}), we have
\begin{equation*}
\partial_{t}u_{a}=\frac{1}{n}u^{ij}u_{aij},
\end{equation*}
\begin{equation*}
\partial_{t}u_{ab}=\frac{1}{n}u^{ij}u_{abij}-\frac{1}{n}u^{ic}u^{jd}u_{aij}u_{bcd},
\end{equation*}
\begin{equation*}\aligned
\partial_{t}u_{kpr}=&\frac{1}{n}u^{ij}u_{kprij}-\frac{1}{n}u^{ia}u^{jb}u_{rab}u_{kpij}\\
&-\frac{1}{n}u^{ic}u^{jd}u_{pcd}u_{krij}-\frac{1}{n}u^{ic}u^{jd}u_{kij}u_{prcd}\\
&+\frac{1}{n}u^{ia}u^{cb}u_{rab}u^{jd}u_{kij}u_{pcd}+\frac{1}{n}u^{ic}u^{ja}u^{db}u_{rab}u_{kij}u_{pcd}.
\endaligned
\end{equation*}
Then
\begin{equation}\aligned\label{e3.1a}
n\partial_{t}\sigma =& 2u^{kl}u^{pq}u^{rs}u^{ij}u_{lqs}u_{kprij}-6u^{kl}u^{pq}u^{rs}u^{ia}u^{jb}u_{lqs}u_{rab}u_{kpij}\\
&+4u^{kl}u^{pq}u^{rs}u^{ia}u^{cb}u^{jd}u_{lqs}u_{rab}u_{kij}u_{pcd}                                                   \\
&-3u^{ka}u^{bl}u^{pq}u^{rs}u^{ij}u_{kpr}u_{lqs}u_{abij}+3u^{ka}u^{bl}u^{pq}u^{rs}u^{ic}u^{jd}u_{kpr}u_{lqs}u_{aij}u_{bcd}.
\endaligned
\end{equation}
By the computation in \cite{LLJ1}, we have
\begin{equation}\aligned\label{e3.2}
u^{ij}\sigma_{ij}=& 2u^{kl}u^{pq}u^{rs}u^{ij}u_{lqs}u_{kprij}+2u^{kl}u^{pq}u^{rs}u^{ij}u_{kpri}u_{lqsj}\\
&-12u^{ka}u^{bl}u^{pq}u^{rs}u^{ij}u_{abi}u_{lqs}u_{kprj}\\
&+6u^{ka}u^{bl}u^{pc}u^{dq}u^{rs}u^{ij}u_{kpr}u_{lqs}u_{abi}u_{cdj}\\
&-3u^{ka}u^{bl}u^{pq}u^{rs}u^{ij}u_{kpr}u_{lqs}u_{abij}\\
&+3u^{kc}u^{ad}u^{bl}u^{pq}u^{rs}u^{ij}u_{kpr}u_{lqs}u_{abi}u_{cdj}\\
&+3u^{ka}u^{bc}u^{dl}u^{pq}u^{rs}u^{ij}u_{kpr}u_{lqs}u_{abi}u_{cdj}.
\endaligned
\end{equation}
At any point $x$, we may assume that $u_{ij}$ is diagonal after a
suitable rotation. So the simplified versions of (\ref{e3.1a}), (\ref{e3.2}) are
\begin{equation*}\aligned
n\partial_{t}\sigma =& 2u^{kk}u^{pp}u^{rr}u^{ii}u_{kpr}u_{kprii}-6u^{kk}u^{pp}u^{rr}u^{ii}u^{jj}u_{kpr}u_{rij}u_{kpij}\\
&+4u^{kk}u^{pp}u^{rr}u^{ii}u^{cc}u^{jj}u_{kpr}u_{ric}u_{kij}u_{pcj}                                         \\
&-3u^{kk}u^{bb}u^{pp}u^{rr}u^{ii}u_{kpr}u_{bpr}u_{kbii}+3u^{kk}u^{bb}u^{pp}u^{rr}u^{ii}u^{jj}u_{kpr}u_{bpr}u_{kij}u_{bij},
\endaligned
\end{equation*}
\begin{equation*}\aligned
u^{ij}\sigma_{ij}=& 2u^{kk}u^{pp}u^{rr}u^{ii}u_{kpr}u_{kprii}+2u^{kk}u^{pp}u^{rr}u^{ii}u_{kpri}u_{kpri}\\
&-12u^{kk}u^{bb}u^{pp}u^{rr}u^{ii}u_{kbi}u_{bpr}u_{kpri}
+6u^{kk}u^{bb}u^{pp}u^{dd}u^{rr}u^{ii}u_{kpr}u_{bdr}u_{kbi}u_{pdi}\\
&-3u^{kk}u^{bb}u^{pp}u^{rr}u^{ii}u_{kpr}u_{bpr}u_{kbii}+3u^{kk}u^{aa}u^{bb}u^{pp}u^{rr}u^{ii}u_{kpr}u_{bpr}u_{abi}u_{kai}\\
&+3u^{kk}u^{bb}u^{dd}u^{pp}u^{rr}u^{ii}u_{kpr}u_{dpr}u_{kbi}u_{bdi}.
\endaligned
\end{equation*}
Let
\begin{equation*}
A=u^{kk}u^{pp}u^{rr}u^{ll}u^{qq}u^{ii}u_{kpr}u_{lqr}u_{kli}u_{pqi},\qquad
\end{equation*}
\begin{equation*}
B=u^{kk}u^{pp}u^{rr}u^{ll}u^{qq}u^{ii}u_{kpr}u_{lpr}u_{kqi}u_{lqi}.\qquad
\end{equation*}
Then, we get
\begin{equation*}\aligned
n\partial_{t}\sigma =& 2u^{kk}u^{pp}u^{rr}u^{ii}u_{kpr}u_{kprii}-6u^{kk}u^{pp}u^{rr}u^{ii}u^{jj}u_{kpr}u_{rij}u_{kpij}\\
&-3u^{kk}u^{bb}u^{pp}u^{rr}u^{ii}u_{kpr}u_{bpr}u_{kbii}+4A+3B, \\
\endaligned
\end{equation*}
\begin{equation*}\aligned
u^{ij}\sigma_{ij}=& 2u^{kk}u^{pp}u^{rr}u^{ii}u_{kpr}u_{kprii}+2u^{kk}u^{pp}u^{rr}u^{ii}u_{kpri}u_{kpri}\\
&-12u^{kk}u^{bb}u^{pp}u^{rr}u^{ii}u_{kbi}u_{bpr}u_{kpri}\\
&-3u^{kk}u^{bb}u^{pp}u^{rr}u^{ii}u_{kpr}u_{bpr}u_{kbii}\\
&+6A+3B+3B.
\endaligned
\end{equation*}
It is easy to verify that
\begin{equation*}
u^{kk}u^{bb}u^{pp}u^{rr}u^{ii}u_{kbi}u_{bpr}u_{kpri}=u^{kk}u^{pp}u^{rr}u^{ii}u^{jj}u_{kpr}u_{rij}u_{kpij}.
\end{equation*}
So  we obtain
\begin{equation}\aligned\label{e3.3}
u^{ij}\sigma_{ij}-n\partial_{t}\sigma=&2u^{kk}u^{pp}u^{rr}u^{ii}u_{kpri}u_{kpri}-6u^{kk}u^{pp}u^{rr}u^{ii}u^{jj}u_{kpr}u_{rij}u_{kpij}\\
&+3B+2A.
\endaligned
\end{equation}
Thus
\begin{equation*}\aligned
&2u^{kk}u^{pp}u^{rr}u^{ii}u_{kpri}u_{kpri}-6u^{kk}u^{pp}u^{rr}u^{ii}u^{jj}u_{kpr}u_{rij}u_{kpij}\\
=&2u^{kk}u^{pp}u^{rr}u^{ii}[u_{kpri}-\frac{1}{2}u^{ll}(u_{kli}u_{plr}+u_{pli}u_{klr}+u_{rli}u_{kpl})]^{2}\\
&-\frac{1}{2}u^{kk}u^{pp}u^{rr}u^{ii}\mid
u^{ll}(u_{kli}u_{plr}+u_{pli}u_{klr}+u_{rli}u_{kpl})\mid^{2}\\
=&2u^{kk}u^{pp}u^{rr}u^{ii}[u_{kpri}-\frac{1}{2}u^{ll}(u_{kli}u_{plr}+u_{pli}u_{klr}+u_{rli}u_{kpl})]^{2}\\
&-\frac{3}{2}B-\frac{6}{2}A\\
\geq&-\frac{3}{2}B-3A.
\endaligned
\end{equation*}
By $B\geq A$ and $\large B\geq\dfrac{1}{n}\sigma^{2}$ (cf. \cite{LLJ1}), (\ref{e3.3}) tells us that
\begin{equation*}\aligned
u^{ij}\sigma_{ij}-n\partial_{t}\sigma &\geq \frac{1}{2}B+B-A\\
&\geq\frac{1}{2n}\sigma^{2}.
\endaligned
\end{equation*}
\end{proof}

\begin{Corollary}\label{c3.3}
Assume $u_{0}(x)$ be a smooth function satisfying Condition A and
\begin{equation}\label{e3.4}
 \sup_{x\in \mathbb{R}^{n}}|D^{3}u_{0}|<+\infty.
 \end{equation}
 Set $\sigma_{0}=\sigma\mid_{t=0}$. Then
\begin{equation}\label{e3.5}
\sup_{x\in \mathbb{R}^{n}}\sigma\leq \frac{\sup_{x\in
\mathbb{R}^{n}}\sigma_{0}}{1+\frac{1}{2n^{2}}\sup_{x\in \mathbb{R}^{n}}\sigma_{0}t},\quad \forall t>0,
\end{equation}
i.e,
\begin{equation}\label{e3.6}
\sup_{x\in \mathbb{R}^{n}}|D^{3}u|^{2}\leq \frac{C\sup_{x\in
\mathbb{R}^{n}}|D^{3}u_{0}|^{2}}{1+\sup_{x\in \mathbb{R}^{n}}|D^{3}u_{0}|^{2}t},\quad \forall t>0,
\end{equation}
where $C$ is positive constant depending only on $n, \lambda,  \Lambda$.
\end{Corollary}
\begin{proof}

By Schauder estimates, as in the proof of
 Proposition \ref{p3.1} (cf. \cite{HB}), we have
\begin{equation*}
\sup_{x\in \mathbb{R}^{n}}\sigma\leq C.
\end{equation*}
Here, $C$ is a positive constant depending only on $n, \lambda,
\Lambda$ and  $ \sup_{x\in \mathbb{R}^{n}}|D^{3}u_{0}|$. Set
$\sigma_{*}=\sigma$ and
\begin{equation*}
\sigma^{*}=\frac{\sup_{x\in
\mathbb{R}^{n}}\sigma_{0}}{1+\frac{1}{2n^{2}}\sup_{x\in \mathbb{R}^{n}}\sigma_{0}t}.
\end{equation*}
In this case, one can verify that
\begin{equation*}
\frac{d}{dt}\sigma^{*}+\frac{1}{2n^{2}}\sigma^{*2}=0,
\end{equation*}
with
\begin{equation*}
\sigma^{*}|_{t=0}=\sup_{x\in \mathbb{R}^{n}}\sigma_{0}.
\end{equation*}
Then by Lemma \ref{l3.1a} we obtain (\ref{e3.5}) and (\ref{e3.6}).
\end{proof}

By now we have proved (\ref{e1.6a}) with an additional condition
(\ref{e3.4}). Using Krylov-Safonov  theory and interior Schauder
estimates of parabolic equations, we need not that  $u_{0}$
satisfies (\ref{e3.4}) for our theorem .

{\bf Proof of Theorem \ref{t1.2a}:}

By Proposition \ref{p3.1}, we have
\begin{equation}\label{e3.7}
 \sup_{x\in \mathbb{R}^{n}}|D^{3}u|_{t=\epsilon_{0}}\leq C,
 \end{equation}
where $C$ is a positive constant depending only on $n, \lambda,
\Lambda$ and $\dfrac{1}{\epsilon_{0}}$. Using Corollary \ref{c3.3},
it follows from (\ref{e3.7}) that we obtain (\ref{e1.6a}).

We will derive high order estimates (\ref{e1.7a}) via the blow up
argument. To do so, by \cite{ACH1}, we  employ a parabolic
scaling now. The remaining proof is routine.  Define
\begin{equation*}
y=\mu(x-x_{0}),\qquad s=\mu^{2}(t-t_{0}),
\end{equation*}
\begin{equation*}
u_{\mu}(y,s)=\mu^{2}[u(x,t)-u(x_{0},t_{0})-Du(x_{0},t_{0})\cdot (x-x_{0})].
\end{equation*}
It is easy to see that
\begin{equation*}
D^{2}_{y}u_{\mu}=D^{2}_{x}u,\qquad \frac{\partial}{\partial s}u_{\mu}=\frac{\partial}{\partial t}u
\end{equation*}
and
\begin{equation*}
D^{l}_{y}u_{\mu}=\mu^{2-l}D^{l}_{x}u
\end{equation*}
for all nonnegative integers $l$. By computing, $u_{\mu}(y,s)$
satisfies
\begin{equation*}\label{1.03}
\left\{ \begin{aligned}\frac{\partial u_{\mu}}{\partial
s}-\frac{1}{n}\ln \det D^{2}u_{\mu}&=0,
& s>0,\quad y\in \mathbb{R}^{n}, \\
 u_{\mu}&=u_{\mu}(y,s)|_{t=t_{0}}, & s=0,\quad y\in \mathbb{R}^{n},
\end{aligned} \right.
\end{equation*}
with
\begin{equation}\label{e3.9}
u_{\mu}(0,0)=Du_{\mu}(0,0)=0.
\end{equation}
Without loss of generality, we prove (\ref{e1.7a}) for $l=4$ only since the statement follows in a similar way for all $l$
by induction on $l$.

Note that
\begin{equation*}
 \sup_{x\in \mathbb{R}^{n}}|D^{4}u|< +\infty, \quad t\geq\epsilon_{0}.
 \end{equation*}
Suppose that $|D^{4}u|^{2}t^{2}$ were not bounded on
$\mathbb{R}^{n}\times [\epsilon_{0},+\infty)$.
 By Lemma 3.5 in \cite{HB1}, there would be a sequence $t_{k}\rightarrow +\infty$, such that
\begin{equation}\label{e3.10}
 2\rho_{k}:=\sup_{x\in \mathbb{R}^{n}}|D^{4}u(x,t_{k})|^{2}t^{2}_{k}\rightarrow +\infty
\end{equation}
and
\begin{equation}\label{e3.11}
 \sup_{x\in \mathbb{R}^{n},t\leq t_{k}}|D^{4}u(x,t)|^{2}t^{2}\leq 2\rho_{k}.
\end{equation}
Then there exists $x_{k}$ such that
\begin{equation}\label{e3.12}
|D^{4}u(x_{k},t_{k})|^{2}t^{2}_{k}\geq \rho_{k}\rightarrow +\infty
\qquad \mathrm{as}\quad t_{k}\rightarrow +\infty.
\end{equation}
Let $(y, Du_{\mu_{k}}(y,s))$ be a parabolic scaling of
$(x,Du(x,t))$ by
 $\displaystyle\mu_{k}=(\frac{\rho_{k}}{t^{2}_{k}})^{\frac{1}{4}}$ at $(x_{k},t_{k})$ for each $k$. Thus $u_{\mu_{k}}(y,s)$
 is a solution of a fully nonlinear parabolic equation
\begin{equation}\label{e3.13}
\frac{\partial u_{\mu_{k}}}{\partial s}-\frac{1}{n}\ln \det
D^{2}u_{\mu_{k}}=0,\quad -\mu^{2}_{k}t_{k} <s\leq 0,\quad y\in
\mathbb{R}^{n}.
\end{equation}
Combining (\ref{e3.10}), (\ref{e3.11}) with (\ref{e3.12}), there
holds
\begin{equation}\label{e3.14}
 |D^{2}_{y}u_{\mu_{k}}|=|D^{2}_{x}u|\leq n\Lambda, \qquad (y,s)\in \mathbb{R}^{n}\times (-\mu^{2}_{k}t_{k}, 0];
\end{equation}
\begin{equation*} \begin{aligned}
\forall y\in \mathbb{R}^{n},\quad
 |D^{3}_{y}u_{\mu_{k}}|^{2}&=\mu^{-2}_{k}|D^{3}_{x}u|^{2}\\
 &\leq \mu^{-2}_{k}t^{-1}_{k}C \\
 &=\rho_{k}^{-\frac{1}{2}}C\rightarrow 0
 \end{aligned}
\end{equation*}
and
\begin{equation}\label{e3.15}
\forall y\in \mathbb{R}^{n},\quad
|D^{4}_{y}u_{\mu_{k}}|^{2}=\mu^{-4}_{k}|D^{4}_{x}u|^{2}\leq 2;
\end{equation}
\begin{equation*}
|D^{4}_{y}u_{\mu_{k}}(0,0)|\geq 1.
 \end{equation*}
Using (\ref{e3.13}), by Schauder estimates, there exists a constant
$C$ depending only on $n, \lambda, \Lambda,
\dfrac{1}{\epsilon_{0}}$, such that for $l\geq 4$,  we derive
\begin{equation}\label{e3.16}
\forall (y,s)\in \mathbb{R}^{n}\times (-\mu^{2}_{k}t_{k}, 0],\quad
|D^{l}_{y}u_{\mu_{k}}|^{2}\leq C.
\end{equation}
Combining (\ref{e3.9}), (\ref{e3.14}), (\ref{e3.15}) and
(\ref{e3.16}) together, a diagonal sequence argument shows that
$u_{\mu_{k}}$ converges subsequently
 and uniformly on compact subsets in $\mathbb{R}^{n}\times (-\infty, 0]$ to a smooth function $u_{\infty}$ with
\begin{equation*}
\forall (y,s)\in \mathbb{R}^{n}\times (-\infty, 0],\quad
|D^{3}_{y}u_{\infty}|=0
\end{equation*}
and
\begin{equation*}
|D^{4}_{y}u_{\infty}(0,0)|\geq 1.
\end{equation*}
It is a contradiction. \qed

\section{self-shrinking solutions to lagrangian mean curvature flow in Pseudo-Euclidean space }
We now describe the relationship between Monge-Amp\`{e}re type
equations (\ref{e1.3}) and the logarithmic Monge-Amp\`{e}re  flow.

A solution $F(\cdot,t)$ of (\ref{e1.7}) is called self-shrinking if
it has the form
\begin{equation}\label{e2.1}
M_{t}=\sqrt{-t}M_{-1}\quad \mathrm{for}\,\, \mathrm{all}\,\, t<0,
\end{equation}
where $M_{t}=F(\cdot,t)$.

Assume that $F(x,t)$ is a self-shrinking  solution of (\ref{e1.7}).
Following Proposition 2.1 in \cite{HB},  $u(x,t)$ satisfies
\begin{equation}\label{e2.31}
\frac{\partial u}{\partial t}-\frac{1}{n}\ln
\det D^{2}u=0,
 \,\,\,t<0,\,\,\,\quad x\in \mathbb{R}^{n}.
\end{equation}
Hence,
\begin{equation*}
D(u(x,t)+tu(\frac{x}{\sqrt{-t}}, -1))=0,
\end{equation*}
$\mathrm{i.e.}$,
\begin{equation}\label{e2.41}
u(x,t)=-tu(\frac{x}{\sqrt{-t}}, -1),\,\,\,t<0.
\end{equation}
Thus combining (\ref{e2.31}), (\ref{e2.41}) and letting $t=-1$, we can
verify that  $u(x,-1)$ satisfies (\ref{e1.3}).

Conversely, if $u(x)$ solves  (\ref{e1.3}) , then using
(\ref{e2.41}), we can obtain a solution $F(x,t)$ to (\ref{e1.7})
which is   shrinking .
 Suppose that $u(x)$ solves
(\ref{e1.3}). Define $$u(x,t)=-tu(\frac{x}{\sqrt{-t}}).$$ One can
easily check the family $M_{t}=\{(x,Du(x,t))|x\in \mathbb{R}^{n}\}$
satisfying (\ref{e2.1}) and we also have
\begin{equation*}
\frac{\partial u}{\partial t}(x,t)=-u(\frac{x}{\sqrt{-t}})+\frac{1}{2}<\nabla u,\frac{x}{\sqrt{-t}}>=\frac{1}{n}\ln\det D^{2}u.
\end{equation*}
In other words, $u(x,t)$ solves the logarithmic gradient flow.  By
the above discussion,  there exists a family $r_{t}$, such that
$F(x,t)=(r_{t}(x),  Du(r_{t}(x),t))$
 is a self-shrinking solution of (\ref{e1.7}).

Based on Theorem \ref{t1.2a} according to the parabolic equation
(\ref{e1.6}), we will prove the following Lemma  by the same methods
in \cite{ACH2}.
\begin{lemma}\label{l3.1}
Let $u :\mathbb{R}^{n}\rightarrow \mathbb{R}$ be a smooth
solution of (\ref{e1.3}) which satisfies condition A. Then $u$ must
be a quadratic polynomial.
\end{lemma}
\begin{proof}
If $u$ is a smooth solution to (\ref{e1.3}), then
\begin{equation*}
v(x,t)=(1-t)u(\frac{x}{\sqrt{1-t}})
\end{equation*}
is a solution to (\ref{e1.6}) for $t\in (0,1)$ with initial data
$u(x)$. Hence  applying Proposition \ref{p3.1} to $v(x,t)$ we show
that  this solution is unique.  By Theorem \ref{t1.2a}, there is
some constant $C$, such that $|D^{3}v(x,t)|\leq C$  for
$t\geq\epsilon_{0}$ and any $x\in \mathbb{R}^{n}$ . But one checks
directly that
\begin{equation*}
D^{3}v(x,t)=\frac{1}{\sqrt{1-t}}D^{3}u(\frac{x}{\sqrt{1-t}}).
\end{equation*}
This implies
\begin{equation*}
|D^{3}u(x)|=|D^{3}u(\frac{x\sqrt{1-t}}{\sqrt{1-t}})|=\sqrt{1-t}|D^{3}v(x\sqrt{1-t},t)|\leq C\sqrt{1-t}
\end{equation*}
for any $x$. It follows that $D^{3}u(x)\equiv 0$ by letting $t\rightarrow 1$. Then $u$ must be a quadratic polynomial.
Lemma \ref{l3.1} is established.
\end{proof}
In fact, using the interior estimated skills (c.f. \cite{GT}), we
can get the upper bound for the second derivatives of solutions of
(\ref{e1.3}) under the condition  (\ref{e1.511}).
\begin{lemma}\label{l3.2}
Let $u :\mathbb{R}^{n}\rightarrow \mathbb{R}$ be a  smooth
strictly convex solution to (\ref{e1.3})  and suppose $\mu(x)$ satisfies
(\ref{e1.511}). Then there exists a positive constant $R_{0}$
depending only on $\mu(x)$, such that
\begin{equation}\label{e5.1}
 D^{2}u(x)\leq C I,\qquad x\in
\mathbb{R}^{n},
\end{equation}
where $C$ is a positive constant depending only on $\mu(x)$ and $\|u\|_{C^{2}(\bar{B}_{R_{0}+1})}$.
$B_{R_{0}}$  is a ball centered at $0$ with radius $R_{0}$ in $\mathbb{R}^{n}$.
\end{lemma}
\begin{proof}
Denote $$u_{i}=\frac{\partial u}{\partial x_{i}},\,\,\, u_{ij}=\frac{\partial^{2}u}{\partial x_{i}\partial x_{j}},\,\,\,
u_{ijk}=\frac{\partial^{3}u}{\partial x_{i}\partial x_{j}\partial x_{k}}, \cdots $$and
$$[u^{ij}]=[ u_{ij}]^{-1},\quad L=u^{ij}\frac{\partial^{2}}{\partial x_{i}\partial x_{j}}.$$
Let $\gamma$ denotes a vector field. Set
$$u_{\gamma}=D_{\gamma}u,\,\,\,u_{\gamma\gamma}=D^{2}_{\gamma\gamma}u.$$
 We will prove that
$$\sup_{x\in \mathbb{R}^{n},\,\gamma\in\mathbb{S}^{n-1}}u_{\gamma\gamma}\leq C.$$
 By (\ref{e1.511}), there is some constant $\lambda>\dfrac{2(n-1)}{n}$  and some constant $R_{0}$, such that
$$|x|^{2}\mu(x)\geq \lambda,$$
for $|x|>R_{0}+1$. One can define a family of smooth function by
$$f_{k}(t)=\left\{ \begin{aligned}
&1 , \quad
\qquad\quad\qquad\qquad\quad\quad\qquad\quad 0\leq t\leq R_{0}, \\
& \varphi \qquad\qquad\quad\quad\qquad\quad\qquad\qquad R_{0}\leq t\leq R_{0}+1,\\
&-k[t^{2}-(R_{0}+1)^{2}]+\frac{3}{4}, \qquad \quad t\geq R_{0}+1,
\end{aligned} \right.$$
where $0<k\leq 1$,  and $(t,\varphi(t))$ is a smooth curve
connecting two points $(R_{0},1)$, $(R_{0}+1,\dfrac{3}{4})$
satisfying $\dfrac{3}{4}\leq \varphi\leq 1$.

We view $u_{\gamma\gamma}$ as a function on
$\mathbb{R}^{n}\times\mathbb{S}^{n-1}$. It is easy to see that
$f_{k}(|x|)u_{\gamma\gamma}$ always attains its maximum at
$$(p,\xi)\in \{(x,\gamma)\in \mathbb{R}^{n}\times\mathbb{S}^{n-1}|f_{k}(|x|)>0\}.$$
  By (\ref{e1.511}), we have  $u_{\gamma\gamma}>0$.  Let
  $$\eta_{k}(x)=f_{k}(|x|),\,\,\, w=\eta_{k}(x)u_{\xi\xi}.$$
Then at $p$,
\begin{equation}\label{e5.2}
0\geq
Lw=u^{ij}(\eta_{k}u_{\xi\xi})_{ij}=u^{ij}(\eta_{k})_{ij}u_{\xi\xi}+2u^{ij}(\eta_{k})_{i}(u_{\xi\xi})_{j}
+\eta_{k}u^{ij}(u_{\xi\xi})_{ij}.
\end{equation}
We assume that
$$p\in\{x\in \mathbb{R}^n||x|> R_{0}+1\}.$$
Then at $p$, the derivative $u_{\xi\xi}$ will be the maximum
eigenvalue of the Hessian $D^{2}u$. By a rotation, we can assume
that $D^{2}u$  is diagonal with $\xi$ as the $x_{1}$ direction. In
this case, $u_{\xi\xi}=u_{11}$. Then at $p$, there holds
\begin{equation*}
(\eta_{k}u_{11})_{j}=0, \,\,\,\,\, j=1,2,\cdots,n.
\end{equation*}
Hence
\begin{equation}\label{e5.3}
(u_{11})_{j}=-u_{11}\frac{(\eta_{k})_{j}}{\eta_{k}},\,\,\,(\eta_{k})_{j}=-\eta_{k}\frac{(u_{11})_{j}}{u_{11}},\,\,\,\,\, j=1,2,\cdots,n.
\end{equation}
Clearly, by (\ref{e5.3}),
\begin{equation}\aligned\label{e5.4}
2u^{ij}(\eta_{k})_{i}(u_{11})_{j}=&u^{11}(\eta_{k})_{1}u_{111}+u^{11}(\eta_{k})_{1}u_{111}+2\sum_{i\neq 1}\frac{(\eta_{k})_{i}u_{11i}}{u_{ii}}\\
=&-u^{11}\frac{(\eta_{k})_{1}(\eta_{k})_{1}}{\eta_{k}}u_{11}-u^{11}\eta_{k}\frac{u^{2}_{111}}{u_{11}}-2\sum_{i\neq 1}\eta_{k}\frac{u^{2}_{11i}}{u_{ii}u_{11}}.
\endaligned
\end{equation}
Let $<\cdot,\cdot>$ be the inner product in $\mathbb{R}^{n}$.
Differentiating  the equation (\ref{e1.3}), we have
\begin{equation*}
\frac{1}{n}u^{ij}u_{ij1}=-\frac{1}{2}u_{1}+\frac{1}{2}<x,Du_{1}>,
\end{equation*}
\begin{equation}\label{e5.5}
\frac{1}{n}u^{ij}u_{11ij}=\frac{1}{n}\sum_{i,j=1}^{n}\frac{u^{2}_{ij1}}{u_{ii}u_{jj}}+\frac{1}{2}<x,Du_{11}>.
\end{equation}
Substituting (\ref{e5.4}), (\ref{e5.5}) into (\ref{e5.2}) and using
$$(\eta_{k})_{i}=-2kx_{i},\,\,\,\,\, (\eta_{k})_{ij}=-2k\delta_{ij},$$
we have, at $p$,
\begin{equation*}\aligned
0\geq & -\frac{1}{n}2k\sum_{i=1}^{n}u^{ii}u_{11}-\frac{1}{n}\frac{(\eta_{k})^{2}_{1}}{\eta_{k}}-
\frac{1}{n}\eta_{k}\frac{u^{2}_{111}}{u^{2}_{11}}-
\frac{1}{n}2\eta_{k}\sum_{i\neq 1}\frac{u^{2}_{11i}}{u_{ii}u_{11}}\\
&+\frac{1}{n}\eta_{k}\sum_{i,j=1}^{n}\frac{u^{2}_{ij1}}{u_{ii}u_{jj}}+\frac{\eta_{k}}{2}<x,Du_{11}>.
\endaligned
\end{equation*}
Note that
\begin{equation*}
\eta_{k}\sum_{i,j=1}^{n}\frac{u^{2}_{ij1}}{u_{ii}u_{jj}}\geq
\eta_{k}\frac{u^{2}_{111}}{u^{2}_{11}} +2\eta_{k}\sum_{i\neq
1}\frac{u^{2}_{11i}}{u_{ii}u_{11}}.
\end{equation*}
Combining the above two inequalities, at $p$, we get
\begin{equation*}
0\geq  -\frac{1}{n}2k\sum_{i=1}^{n}u^{ii}u_{11}-\frac{1}{n}\frac{(\eta_{k})^{2}_{1}}{\eta_{k}}
+\frac{\eta_{k}}{2}<x,Du_{11}>.
\end{equation*}
In view of (\ref{e5.3}),
\begin{equation*}
\frac{\eta_{k}}{2}<x,Du_{11}>=-\frac{u_{11}}{2}<x,D \eta_{k}>.
\end{equation*}
Then at $p$,
\begin{equation*}
\frac{(\eta_{k})^{2}_{1}}{\eta_{k}}\geq  -2k\sum_{i=1}^{n}u^{ii}u_{11}
-n\frac{u_{11}}{2}<x,D \eta_{k}>.
\end{equation*}
Using $u_{ii}\geq \dfrac{\lambda}{|x|^{2}}$ for $i\geq 2$, we deduce from the above that
\begin{equation*}
\frac{4k^{2}x^{2}_{1}}{\eta_{k}}\geq
-2k-\frac{2k(n-1)}{\lambda}|x|^{2}u_{11}+nk|x|^{2}u_{11},
\end{equation*}
i.e., at $p$,
\begin{equation*}
\frac{4kx^{2}_{1}+2\eta_k}{n|x|^{2}-\frac{2(n-1)}{\lambda}|x|^{2}}\geq  \eta_{k}u_{11}.
\end{equation*}
Thus   if $p\in\{x\in\mathbb{R}^{n}||x|> R_{0}+1\}$,  then there holds
\begin{equation}\label{e5.6}
\max_{x\in\mathbb{R}^{n},\gamma\in
\mathbb{S}^{n-1}}\eta_{k}u_{\gamma\gamma}\leq\frac{4k\lambda
x^{2}_{1}+2\lambda\eta_k}{(\lambda-\frac{2(n-1)}{n})n|x|^{2}}\leq
\frac{6\lambda }{(\lambda-\frac{2(n-1)}{n})n}.
\end{equation}
 And  if $p\in\{x\in\mathbb{R}^{n}||x|\leq R_{0}+1\}$, then
\begin{equation}\label{e5.7}
\max_{x\in\mathbb{R}^{n},\gamma\in
\mathbb{S}^{n-1}}\eta_{k}u_{\gamma\gamma}\leq
\|u\|_{C^{2}(\bar{B}_{R_{0}+1})}.
\end{equation}
From (\ref{e5.6}) and (\ref{e5.7}), we obtain
\begin{equation}\label{e5.712}
\max_{x\in\mathbb{R}^{n},\gamma\in
\mathbb{S}^{n-1}}\eta_{k}u_{\gamma\gamma}\leq \frac{6\lambda
}{(\lambda-\frac{2(n-1)}{n})n}+\|u\|_{C^{2}(\bar{B}_{R_{0}+1})}.
\end{equation}
For any fixed $x\in \mathbb{R}^{n}$ and $\gamma\in
\mathbb{S}^{n-1}$, let $k$ converges to $0$,  then
$$\frac{3}{4}u_{\gamma\gamma}\leq \frac{6\lambda }{(\lambda-\frac{2(n-1)}{n})n}+\|u\|_{C^{2}(\bar{B}_{R_{0}+1})}.$$
So we obtain
$$u_{\gamma\gamma}\leq \frac{24\lambda }{(3\lambda-\frac{6(n-1)}{n})n}+\frac{4}{3}\|u\|_{C^{2}(\bar{B}_{R_{0}+1})}$$
and Lemma \ref{l3.2} is established.
\end{proof}
{\bf Proof of Theorem \ref{t1.1}:}

Introduce the Legendre transformation of $u$,
\begin{equation*}
y_{i}=\frac{\partial u}{\partial x_{i}}
,\,\,i=1,2,\cdots,n,\,\,\,u^{*}(y_{1},\cdots,y_{n}):=\sum_{i=1}^{n}x_{i}\frac{\partial u}{\partial x_{i}}-u(x).
\end{equation*}
In terms of $y_{1},\cdots,y_{n}, u^{*}(y_{1},\cdots,y_{n})$, one can easily check that
$$\frac{\partial^{2} u^{*}}{\partial y_{i}\partial y_{j}}=[\frac{\partial^{2} u}{\partial x_{i}\partial x_{j}}]^{-1}.$$
Thus, in view of (\ref{e5.1}), $$D^{2}u^{*}\geq \frac{1}{C}I.$$ And
the PDE (\ref{e1.3}) can be rewritten as
\begin{equation*}
\det D^{2}u^{*}=\exp\{n(-u^{*}+\frac{1}{2}\sum_{i=1}^{n}y_{i}\frac{\partial u^{*}}{\partial y_{i}})\}.
\end{equation*}
Using Lemma \ref{l3.2}, we have
 $$D^{2}u^{*}\leq CI.$$
So
$$\frac{1}{C}I\leq D^{2}u\leq CI.$$
An application of  Lemma \ref{l3.1} yields the desired result. \qed

\section{self-shrinking solutions to lagrangian mean curvature flow in Euclidean space}
In this section, first we present the proof of Theorems \ref{t1.6}.
Then, by the Lewy rotation, we obtain  Theorem \ref{c1.7}.

{\bf Proof of Theorem \ref{t1.6}:}

For $x\in\mathbb{R}^n$, let
\begin{eqnarray*}\label{e4.1}
\eta_k(x)=\left\{\begin{matrix}1& |x|\leqq R_0\\
-k(|x|^2-R_0^2)+1 & |x|\geqq R_0\end{matrix}\right..
\end{eqnarray*}
Here $R_0$ and $k<1$ be two positive constants which we will
determine later. Similar to \cite{Y}, we denote
\begin{eqnarray*}\label{e4.2}
g_{ij}=\delta_{ij}+\sum_{k=1}^nu_{ik}u_{kj}.
\end{eqnarray*}
By (\ref{e1.10}), we have
\begin{eqnarray*}\label{e4.3}
-I\leq D^2u\leq I.
\end{eqnarray*}
We set
\begin{eqnarray*}\label{e4.4}
\phi(x)=\eta_ke^{\alpha\ln\det g},
\end{eqnarray*}
where $\alpha$ is a positive constant which will be determined
later. Assume that at the point $p$, $\phi$ attains its maximum
value. Obviously, at $p$, $\eta_k(p)>0$. If
$$p\in \{X\in\mathbb{R}^n||X|> R_0\},$$ then
\begin{eqnarray}\label{e4.5}
(\eta_k)_{ij}=-2k\delta_{ij},\ \ (\eta_k)_i=-2kx_i.
\end{eqnarray}
At $p$, we get
\begin{eqnarray}\label{e4.6}
D \eta_k+\eta_k\alpha D \ln \det g=0,
\end{eqnarray}
and
\begin{equation*}\aligned
0\geq& g^{ij}\phi_{ij}\\
=&g^{ij}(\eta_k)_{ij}e^{\alpha\ln\det g}+2g^{ij}(\eta_k)_i(e^{\alpha\ln\det g})_j+g^{ij}\eta_k(e^{\alpha\ln\det g})_{ij}\\
=&e^{\alpha\ln\det g}[g^{ij}(\eta_k)_{ij}+2g^{ij}(\eta_k)_i(\alpha\ln\det g)_j+\eta_kg^{ij}(\alpha\ln\det g)_{ij}\\
&+\eta_k g^{ij}(\alpha\ln\det g)_i(\alpha\ln\det g)_j].
\endaligned
\end{equation*}
We pick a coordinate system satisfying $u_{ij}=u_{ii}\delta_{ij}$ at
$p$. Then, inserting (\ref{e4.5}) and (\ref{e4.6}) to the above
inequality, at $p$, we get
\begin{equation}\label{e4.8}
0\geq-2k\sum_i\frac{1}{1+u_{ii}^2}+\eta_kg^{ij}(\alpha\ln\det
g)_{ij}-\eta_k g^{ij}(\alpha\ln\det g)_i(\alpha\ln\det g)_j.
\end{equation}
Differentiating (\ref{e1.9}) twice, we have
\begin{eqnarray*}
g^{lk}u_{lki}&=&-\frac{u_i}{2}+\frac{1}{2}<x,D u_i>,\\
g^{lk}u_{lkij}&=&g^{lm}g^{nk}u_{lki}\sum_{s=1}^n(u_{msj}u_{sn}+u_{ms}u_{snj})+\frac{1}{2}<x,D
u_{ij}>.
\end{eqnarray*}
Similar to Lemma 2.1 in \cite{Y}, we arrive at, at $p$,
\begin{eqnarray}\label{e4.9}
&&g^{ij}(\ln\det g)_{ij}\\
&=&g^{ij}(g^{ab})_i(g_{ab})_j+g^{ij}g^{ab}(g_{ab})_{ij}\nonumber\\
&=&\sum_{i,a,b=1}^n-g^{ii}g^{aa}g^{bb}u_{abi}^2(u_{aa}+u_{bb})^2+\sum_{a,b=1}^ng^{ij}g^{ab}u_{abij}(u_{aa}+u_{bb})\nonumber\\
&&+\sum_{i,k,a=1}^n2g^{aa}g^{ii}u_{aki}^2\nonumber\\
&=&2\sum_{a,b,c=1}^ng^{aa}g^{bb}g^{cc}u_{abc}^2(1+u_{aa}u_{bb})+\sum_{a,b,c=1}^ng^{aa}u_{aa}<x,D
u_{aa}>\nonumber\\
&=&2\sum_{a,b,c=1}^ng^{aa}g^{bb}g^{cc}u_{abc}^2(1+u_{aa}u_{bb})+\frac{1}{2}<x,
D \ln\det g>\nonumber.
\end{eqnarray}
Inserting the above equality into (\ref{e4.8}) and combining
(\ref{e4.5}) with (\ref{e4.6}), we obtain
\begin{eqnarray}
0&\geq&-2kn+\eta_k\alpha\frac{1}{2}<x,D\ln\det g>+2\eta_k\alpha
\sum_{a,b,c=1}^ng^{aa}g^{bb}g^{cc}u_{abc}^2(1+u_{aa}u_{bb})\nonumber\\
&&-4\eta_k\alpha^2\sum_{i=1}^ng^{ii}(\sum_{a=1}^ng^{aa}u_{aa}u_{aai})^2\nonumber\\
&\geq&-\frac{1}{2}<x,D \eta_k>-2kn+2\eta_k\alpha
\sum_{a,b,c=1}^ng^{aa}g^{bb}g^{cc}u_{abc}^2(1+u_{aa}u_{bb})\nonumber\\
&&-4n^2\eta_k\alpha^2\sum_{a,b=1}^ng^{bb}g^{aa}g^{aa}u^2_{aa}u_{aab}^2\nonumber\\
&\geq&k(|x|^2-2n)+2\eta_k\alpha(1-2n^2\alpha)\sum_{a,b=1}^ng^{bb}g^{aa}g^{aa}u^2_{aa}u_{aab}^2\nonumber.
\end{eqnarray}
If we take
\begin{eqnarray}\label{e4.10}
R_0>\sqrt{2n},\text{  and  } \alpha<\frac{1}{2n^2},
\end{eqnarray}
we have a contradiction.

Assume the function $\ln\det g$ is not  constant in $\mathbb{R}^n$.
Then there is a ball $B_{R_0}$ centered at $0$ with radius $R_0$
satisfying (\ref{e4.10}), such that the function $\ln\det g$ is not
a constant in $B_{R_0}$. Suppose that $\ln\det g$ attains its
maximum value in $B_{R_0}$. Applying strong maximum principle to
(\ref{e4.9}), we obtain $\ln \det g$ is a constant. This is a
contradiction. Hence $\ln\det g$ attains its maximum value only on
the boundary $\partial B_{R_0}$. Similarly, in $B_{\sqrt{R_0^2+1}}$,
$\ln \det g$ also attains its maximum value only on the boundary
$\partial B_{\sqrt{R_0^2+1}}$. We assume that the points $p_1$ and
$p_2$ be maximum value points with respect to $\partial B_{R_0}$ and
$\partial B_{\sqrt{R_0^2+1}}$, namely,
\begin{eqnarray}
\max_{\overline{B}_{R_0}}\ln\det g=\ln\det g (p_1),\nonumber\\
\max_{\overline{B}_{(R_0^2+1)^{1/2}}}\ln\det g=\ln\det g
(p_2).\nonumber
\end{eqnarray}
Then
\begin{eqnarray}
\ln\det g (p_1)\leq\ln\det g (p_2).\nonumber
\end{eqnarray}
But the equality can not hold. In fact, if the equality holds, then
the function $\ln\det g$ achieves  its maximum value in the interior
of the domain $B_{\sqrt{R_0^2+1}}$. This is a contradiction. So we
can choose $k$ sufficiently small such that
\begin{eqnarray}
\phi(P_1)=(\det g)^{\alpha}(p_1)<(1-k)(\det
g)^{\alpha}(p_2)=\phi(p_2).\nonumber
\end{eqnarray}
This means that, for fixed $u$, we can choose suitable $k$ such that
the maximum value of $\phi$ only occurs in the set
$$\{X\in\mathbb{R}^n||X|>R_0\}.$$ But we have proved that it is
impossible. Thus the discussion implies the function $\ln\det g$
is a constant. So by (\ref{e4.9}), we have
\begin{eqnarray}
g^{aa}g^{bb}g^{cc}u_{abc}^2(1+u_{aa}u_{bb})=0.\nonumber
\end{eqnarray}
Now we can use the same argument of Proposition 2.1 in \cite{Y}. We obtain
\begin{eqnarray}
u^2_{abc}(1+u_{aa}u_{bb})=u^2_{abc}(1+u_{bb}u_{cc})=u^2_{abc}(1+u_{cc}u_{aa}).\nonumber
\end{eqnarray}
Observe that one of $u_{aa}u_{bb},u_{bb}u_{cc}$ and $u_{cc}u_{aa}$
must be nonnegative, we get, at every point,
$$u_{abc}=0.$$ Consequently, $u$ is a quadric polynomial.
\qed
\begin{proposition}\label{p4.1}
Assume that  $u$  be a smooth solution to (\ref{e1.9}) and  $D^{2}u$  satisfies
\begin{equation}\label{e4.11}
 D^{2}u \geq 0.
 \end{equation}
 Set Lewy rotation\cite{Y},
\begin{equation}\label{e4.12}
\left\{ \begin{aligned}\bar{x}&=\frac{x+D u(x)}{\sqrt{2}} \\
D\bar{u}(\bar{x})&=\frac{-x+D u(x)}{\sqrt{2}}
\end{aligned} \right..
\end{equation}
 Then  $\bar{u}$ is a smooth solution to (\ref{e1.9}) and $D^2\bar{u}$ satisfies  (\ref{e1.10}).
\end{proposition}
{\bf Proof:} Suppose that
$$F=\arctan(\lambda_1)+\cdots+\arctan(\lambda_n),\,\,\,G=-u+\frac{1}{2}<x,D
u>,$$
$$x=(x_{1},\cdots,x_{n}),\,\,\,
\frac{\partial F}{\partial x}=(\frac{\partial F}{\partial x_{1}},\cdots,\frac{\partial F}{\partial x_{n}}).$$
Then
\begin{equation}\label{e4.13}
\frac{\partial F}{\partial \bar{x}}=\frac{\partial F}{\partial
x}\frac{\partial x}{\partial \bar{x}}.
\end{equation}
By (\ref{e4.12}), we get
\begin{equation}\label{e4.14}
\left\{ \begin{aligned}x=\frac{\bar{x}-D \bar{u}(\bar{x})}{\sqrt{2}} \\
D u(x)=\frac{\bar{x}+D \bar{u}(\bar{x})}{\sqrt{2}}
\end{aligned} \right..
\end{equation}
So
\begin{equation}\label{e4.15}
\frac{\partial x}{\partial \bar{x}}=\frac{I-D^{2}
\bar{u}(\bar{x})}{\sqrt{2}}.
\end{equation}
By  (\ref{e1.9}) and (\ref{e4.14}), we have
\begin{equation} \begin{aligned}\label{e4.16}
\frac{\partial F}{\partial x}=\frac{\partial G}{\partial x}&=-\frac{1}{2}D u+\frac{1}{2}xD^{2} u\\
&=-\frac{\bar{x}+D
\bar{u}(\bar{x})}{2\sqrt{2}}+\frac{1}{4}(\bar{x}-D
\bar{u}(\bar{x}))(I+D^{2} \bar{u}(\bar{x})) \frac{\partial
\bar{x}}{\partial x}.
\end{aligned}
\end{equation}
Using (\ref{e4.13}), (\ref{e4.15}), (\ref{e4.16}) and
$$\frac{\partial \bar{x}}{\partial x}\frac{\partial x}{\partial \bar{x}}=I,$$
we obtain
\begin{equation} \begin{aligned}\label{e4.17}
\frac{\partial F}{\partial \bar{x}}&=-\frac{1}{4}(\bar{x}+D
\bar{u}(\bar{x}))(I-D^{2} \bar{u}(\bar{x}))
+\frac{1}{4}(\bar{x}-D \bar{u}(\bar{x}))(I+D^{2} \bar{u}(\bar{x}))\\
&=-\frac{1}{2}D \bar{u}+\frac{1}{2}\bar{x}D^{2}\bar{u}.
\end{aligned}
\end{equation}
From (\ref{e4.12}), we see that
\begin{equation}\label{e4.18}
D^{2}\bar{u}=(I+D^{2}u)^{-1}(-I+D^{2}u).
\end{equation}
Hence,
\begin{equation}\label{e4.19}
\arctan(\lambda_1)+\cdots+\arctan(\lambda_n)=\frac{n\pi}{4}+(\arctan(\bar{\lambda}_1)+\cdots+\arctan(\bar{\lambda}_n)).
\end{equation}
Set
$$\bar{F}=\arctan(\bar{\lambda}_1)+\cdots+\arctan(\bar{\lambda}_n),\,\,\,\bar{G}=-\bar{u}+\frac{1}{2}<\bar{x},D
\bar{u}>.$$ Combining  (\ref{e4.17}) with (\ref{e4.19}), we obtain
$$\frac{\partial \bar{F}}{\partial \bar{x}}=-\frac{1}{2}D \bar{u}+\frac{1}{2}\bar{x}D^{2}\bar{u}=
\frac{\partial \bar{G}}{\partial \bar{x}}.$$
 By (\ref{e4.18}),
$$-I\leq D^{2}\bar{u}\leq I.$$
This completes the proof of Proposition \ref{p4.1}.
\qed

{\bf Proof of Theorem \ref{c1.7}:}

Case 1. Assume that  $u$  be a smooth convex solution to
(\ref{e1.9}). By Lewy rotation (\ref{e4.12}) in Proposition
\ref{p4.1},  $\bar{u}$ is a smooth solution to (\ref{e1.9}) and
$D^2\bar{u}$ satisfies  (\ref{e1.10}). Using Theorem  \ref{t1.6},
$D^2\bar{u}$ must be a constant matrix.   From (\ref{e4.18}), we
deduce that $u$ is a quadric polynomial.

Case 2. Assume  that $u$  is a smooth concave solution to (\ref{e1.9}).  Set $u^{*}=-u$,
then $u^{*}$ must be a quadric polynomial by case 1.

So we have the desired results.
\qed

\vspace{5mm}
{\bf Acknowledgment:} After this paper was submitted, we learned that a
similar result  was obtained by A. Chau, J.Y. Chen and Y. Yuan \cite{ACY}, in which the Bernstein type theorem of (\ref{e1.9}) holds
without any assumption.
 We wish to express our sincere gratitude to Professor Y.L. Xin and Professor J.G. Bao for
 their  valuable suggestions and comments. We also would like to thank  referees for useful comments, which improved the paper.


\begin{thebibliography}{DU}

  \bibitem{ACH1} A. Chau, J.Y. Chen, W.Y. He,
 {\it Lagrangian mean curvature flow for entire lipschitz graphs. }
 arXiv: 0902.3300.

 \bibitem{ACH2} A. Chau, J.Y. Chen, W.Y. He,
{\it Entire self-similar solutions to Lagrangian mean curvature flow}.  arXiv: 0905.3869.

\bibitem{ACY} A. Chau, J.Y. Chen, Y. Yuan,
{\it Rigidity of Entire self-shrinking solutions to curvature flows.}
www.math.washington.edu/~yuan/papers/index.html, to appear in J. Reine Angew. Math.


\bibitem{LX}  A.M. Li, R.W. Xu, {\it A rigidity theorem for an affine K\"{a}hler-Ricci flat graph}.
Result. Math. (2009), 1-24.

\bibitem{P} A.V. Pogorelov, {\it On the improper convex affine hyperspheres}. Geom. Dedi. {\bf 1} (1972), 33-46.

\bibitem{GT} D. Gilbarg, N.S. Trudinger, {\it Elliptic partial differential equations
of second order, }  Second edition, Grundlehren der Mathematischen
Wissenschaften, 224, Berlin: Springer-Verlag, 1998.


\bibitem{C} E. Calabi, {\it Improper affine hypersurfaces of convex type and a generalization of a theorem by K. J\"{o}rgens}.
Michigan Math.J. {\bf 5} (1958), 105-126.

\bibitem{BC}  J.G. Bao, J.Y. Chen, B. Guan, M. Jin {\it Liouville property and regularity of a Hessian quotient equation}.
Amer. J. Math. 125 (2003), 301-316.


\bibitem{Nit} J.C.C. Nitsche, {\it Elementary proof of Bernstein's theorem on minimal surfaces}.
 Ann. Math.
{\bf 66} (1957), 543-544.

\bibitem{JX} J. Jost, Y.L. Xin, {\it Some aspects of the global geometry of entire space-like submanifolds}. Rusult. Math. {\bf 40} (2001), 233-245.

\bibitem{J} K. J\"{o}rgens, {\it ¨¹ber die L\"{o}sungen der Differentialgleichung rt-$s^{2}$=1}.
Math. Ann. {\bf 127} (1954), 130-134.

\bibitem{KK} K. Groh, M. Schwarz, K. Smoczyk, K. Zehmisch, {\it Mean curvature flow of monotone Lagrangian submanifolds}.
Math. Z. {\bf 257} (2007), 295-327.

\bibitem{KM}K. Smoczyk and M.T. Wang,
{\it Mean curvature flows of Lagrangian submanifolds with convex
potentials}, J. Diff Geom. , {\bf 62}(2002), 243-257.


 \bibitem{KT}K.S. Tso, {\it On a real Monge-Ampere functional}.  Invent. Math. {\bf 101}(1990), 425-448.


\bibitem{LL} L. Caffarelli, Y.Y. Li, {\it An extension to a theorem of J\"{o}rgens, Calabi, and Pogorelov}.
Comm. Pure Appl. Math. {\bf 56} (2003), 549-583.

\bibitem{LLJ1} L. Caffarelli, L. Nirenberg,  J. Spruck,
{\it The Dirichlet problem for nonlinear second-order elliptic equations I. Monge-ampere equation,} Comm. Pure. Appl. Math.
{\bf 37}(1984), 369-402.


\bibitem{HL}  R. Harvey, H.B. Lawson , {\it Calibrated geometry}.
Acta. Math. 148(1982),47-157

\bibitem{HB} R.L. Huang, {\it Lagrangian mean Curvature flow In pseudo-Euclidean Space}, Chin. Ann. Math. - Series B, 32(2011), 187-200.

\bibitem{HB1} R.L. Huang, J.G. Bao, {\it The blow up analysis of  the general curve shortening flow,}
 arXiv:0908.2036.

\bibitem{CM}
T.H. Colding, W.P. Minicozzi, {\it Generic mean curvature flow I;
generic singularities}. arXiv: 0908.3788.



\bibitem{Y-M}Y. Giga, S. Goto, H. Ishii, M.H. Sato,
{\it Comparison principle and convexity preserving properties for singular degenerate parabolic equations on unbounded domains,}
Indiana. Univ. Math. J.
{\bf 40}(1991), 443-470.



\bibitem{Y}Y. Yuan,
{\it A Bernstein problem for special Lagrange equations},
Invent. Math. {\bf 150}(2002), 117-125.









































\end{thebibliography}
\end{document}